\documentclass[reqno]{amsart}

\usepackage{amsfonts}
\usepackage{amsmath,amssymb,amsthm}
\usepackage{mathtools}
\usepackage{etoolbox}
\usepackage{color}
\usepackage[english]{babel}
\usepackage{hyperref}
\usepackage{microtype}
\usepackage{fullpage}
\usepackage{graphicx}
\usepackage{floatrow}
\usepackage[utf8]{inputenc}
\usepackage[T1]{fontenc}
\usepackage{lmodern}
\usepackage{mathrsfs}
\usepackage{microtype}
\DisableLigatures{encoding = *, family = * }
\usepackage{xcolor}
\usepackage{tikz}
\usepackage{pgfplots}
\usepackage{pgfplotstable}
\usepackage[rightcaption]{sidecap}
\usepackage{microtype}
\usepackage[T1]{fontenc}
\usepackage{lmodern}
\usepackage{subcaption}

\DisableLigatures{encoding = *, family = * }

\newtheorem{theorem}{Theorem}[section]
\newtheorem{definition}[theorem]{Definition}

\newtheorem{lemma}[theorem]{Lemma}

\newcommand{\func}[1]{\operatorname{#1}}

\numberwithin{equation}{section}

\title{Blow-up results for a logarithmic pseudo-parabolic $p(.)$-Laplacian type equation} 

\author{Belhaoues Razik\textsuperscript{\,$\ast$}}  

\address{\textsuperscript{$\ast$}\, Department of technical sciences, Laboratory of Pure and Applied Mathematics, University of Laghouat 03000, Algeria.
	\newline \indent \textsuperscript{$\dagger$}\, Chair of Computational Mathematics, DeustoTech, University of Deusto, Avenida de las Universidades 24, 48007 Bilbao, Basque Country, Spain.}
	
\email{r.belhaoues@lagh-univ.dz}
\thanks{This project has received funding from the European Research Council (ERC) under the European Union's Horizon 2030 research and innovation programme (grant agreement NO: 101096251-CoDeFeL). UB was partially supported by the Grant PID2023-146872OB-I00-DyCMaMod of MICIU (Spain) and by the COST Actions ``CA24122 - multiscale Stochastics, Patterns, and Analysis of Combinatorial Environments'' and ``CA24136 - Interactions between Control Theory and Machine Learning''}

\author{Umberto Biccari\textsuperscript{\,$\dagger$}}
\email{umberto.biccari@deusto.es}

\author{Abita Rahmoune\textsuperscript{\,$\ast$}}
\email{abitarahmoune@yahoo.fr}

\keywords{Pseudo-parabolic equation,  Logarithmic nonlinearity,  Variable nonlinearity, Blow-up time, Bounds of the blow-up time}
\subjclass[2010]{Mathematics Subject Classification 2000: 35B44, 35K55, 35B40}

\begin{document}

\begin{abstract}
In this paper, we consider an initial-boundary value problem for the following mixed pseudoparabolic $p(.)$-Laplacian type equation with logarithmic nonlinearity:
\begin{equation*}
	u_t-\Delta u_t-\func{div}\left(\left\vert \nabla u\right\vert^{p(.)-2}\nabla u\right) =|u|^{q(.)-2}u\ln(|u|), \quad (x,t)\in\Omega\times(0,+\infty),
\end{equation*}
where $\Omega\subset\mathbb{R}^n$ is a bounded and regular domain, and the variable exponents $p(.)$ and $q(.)$ satisfy suitable regularity assumptions. By adapting the first-order differential inequality method, we establish a blow-up criterion for the solutions and obtain an upper bound for the blow-up time. Besides, we show that blow-up may be prevented under appropriate smallness conditions on the initial datum, in which case we also establish decay estimates in the $H_0^1(\Omega)$-norm as $t\to+\infty$. This decay result is illustrated by a two-dimensional numerical example.
\end{abstract}

\maketitle

\section{Introduction}

Let $\Omega\subset\mathbb{R}^n$ be a bounded domain with smooth boundary $\partial\Omega$. In this paper, we look at the following pseudo-parabolic $p(.)$-Laplacian equation with logarithmic nonlinearity:
\begin{equation}\label{eq:main_eq}
	\begin{cases}
		u_t-\Delta u_t-\operatorname{div}\left(\left\vert \nabla u\right\vert^{p(.)-2}\nabla u\right) =|u|^{q(.)-2}u\ln (|u|), & (x,t)\in \Omega\times(0,+\infty)
		\\
		u(x,t)=0, &(x,t)\in \partial\Omega\times(0,+\infty)
		\\
		u(x,0)=u_0(x)\neq 0, & x\in \Omega
	\end{cases}
\end{equation}
In \eqref{eq:main_eq}, $p(.)$ and $q(.)$ are given measurable functions on $\overline{\Omega}$ satisfying
\begin{align*}
	2\leq p_1\leq p(x) \leq p_2<q_1\leq q(x) \leq q_2 < p^{\ast }(x),
\end{align*}
where we have denoted
\begin{displaymath}
	\begin{array}{ll}
		p_1:=\underset{x\in\Omega}{\mbox{ess inf}}\,p(x), & \quad p_2=\underset{x\in\Omega}{\mbox{ess sup}}\,p(x)
		\\[15pt]
		q_1:=\underset{x\in\Omega}{\mbox{ess inf}}\,q(x), & \quad q_2=\underset{x\in\Omega}{\mbox{ess sup}}\,q(x).
	\end{array}
\end{displaymath}
and
\begin{align}\label{eq:pAst}
	p^{\ast}(x):= \begin{cases} \displaystyle\frac{np(x)}{n-p(x)}, & \mbox{ if } n>p_2, \\ +\infty, & \mbox{ if } n\leq p_2 \end{cases}.
\end{align}

Moreover, we assume that $p(.)$ and $q(.)$ satisfy the following Zhikov-Fan uniform local continuity condition: there exist a constant $M>0$ such that
\begin{equation}\label{03}
	\left\vert q(x)-q(y) \right\vert \leq \frac{M}{\left\vert \log \left\vert x-y\right\vert \right\vert }, \quad\mbox{ for all } (x,y)\in\Omega\times\Omega\;\mbox{ with }\left\vert x-y\right\vert <\frac 12.
\end{equation}

Finally, the initial datum in \eqref{eq:main_eq} is assumed to be any given function $u_0\in W_0^{1,p(.)}(\Omega)$, where $W_0^{1,p(.)}(\Omega)$ denotes the generalization to the variable exponent case of the Sobolev space $W_0^{1,p}(\Omega)$. We refer to Section \ref{sec:2} for more detail.

The nonlinear term $\func{div}\left(\left\vert \nabla u\right\vert^{p(.)-2}\nabla u\right)$ in \eqref{eq:main_eq} is the so-called $p(.)$-Laplacian operator, which is sometimes symbolized by
\begin{align*}
	\Delta_{p(.)}u=\mbox{div}\left(\left\vert \nabla u\right\vert^{p(.)-2}\nabla u\right)
\end{align*}
and is a generalization of the well-known $p$-Laplacian, that corresponds to taking $p(.)\equiv p\in\mathbb{R}$ constant.

This operator can be prolonged to a monotone operator between the space $W_0^{1,p(.)}(\Omega)$ and its dual space $W^{-1,p^{\prime}(.)}(\Omega)$ as follows
\begin{equation*}
	\begin{cases}
		-\Delta_{p(.)}u:W_0^{1,p(.)}(\Omega)\rightarrow W^{-1,p^{\prime}(.)}(\Omega)
		\\[7pt]
		\displaystyle\langle -\Delta_{p(.)}u,\phi \rangle_{p(.)} = \int_{\Omega}\left\vert \nabla u\right\vert^{p(x)-2}\nabla u\nabla\phi\,\mathrm{d}x,\quad 2\leq p_1\leq p(x)\leq p_2<+\infty,
	\end{cases}
\end{equation*}
where $\langle.,.\rangle_{p(.)}$ indicates the duality pairing between $W_0^{1,p(.)}(\Omega)$ and $W^{-1,p^{\prime }(.)}(\Omega)$, and $p'(.)$ denotes the conjugate exponent such that
\begin{align*}
	\frac{1}{p(x)}+\frac{1}{p^{\prime }(x)}=1.
\end{align*}

Nonlinear pseudo-parabolic equations like \eqref{eq:main_eq} arise in the description of several problems in hydrodynamics, thermodynamics, nonlinear elasticity, image processing and filtration theory (see \cite{Korpusov1,Dzektser,Alaoui,Korpusov2,Korpusov3}). Furthermore, we can mention interesting applications in the study of electrorheological fluids, whose viscosity depends on the electric field in the fluid itself (see \cite{Acerbi4,Diening4,Diening3,Diening1,Halsey,Ruzicka4}).

In the case in which the exponents $p(.)$ and $q(.)$ are constants, that is when $p(.)\equiv p\in\mathbb R$ and $q(.)\equiv q\in\mathbb R$, the corresponding PDE
\begin{equation}\label{eq:main_eqConst}
	u_t-\Delta u_t-\func{div}\left(\left\vert \nabla u\right\vert^{p-2}\nabla u\right) =|u|^{q-2}u\ln (|u|),\quad (x,t)\in \Omega\times(0,+\infty)
\end{equation}
has been largely studied by the mathematical community.

In \cite{Showalter,Ting}, the linear version of \eqref{eq:main_eqConst} (i.e. with the right-hand side equal to zero) has been considered in the case $p=2$ in which the $p$-Laplacian $-\Delta_p$ becomes the standard Laplace operator $-\Delta$. In particular, results of existence, uniqueness and regularity of the solutions were obtained.

Later on, these results have been extended to the general case of \eqref{eq:main_eqConst}. In several papers, the asymptotic behavior of weak solutions to \eqref{eq:main_eqConst} with initial datum $u_0\in W^{1,p}_0(\Omega)$ has been studied. In more detail,
\begin{itemize}
	\item The case $p=q>2$ has been considered in \cite{Nhan}.
	\item The case $1<p=q<2$ has been considered in \cite{YANGCAO}.
	\item The case $2<p<q<p(1+\frac{2}{n})$ has been considered in \cite{Wang2018}.
	\item The case $1<p\leq q<p^\ast$, with $p^\ast$ as in \eqref{eq:pAst}, has been considered in \cite{Ding}.
\end{itemize}

In all the aforementioned references, the global-in-time existence or the finite time blow-up of such solutions has been characterized in terms of the values of the exponents $p$ and $q$. We shall stress, however, that the methodology employed in \cite{Wang2018,Nhan,YANGCAO} to prove the blow-up of weak solutions to \eqref{eq:main_eqConst} does not allow the authors to address the relevant issue of estimating the blow-up time and rate (we refer to  \cite{Liu,Payne2010,Payne10,Payne2006,Payne2007,Song} for the analysis of these issues in the case of other interesting non-linear evolution equations).
A partial answer to these relevant questions has been recently given in \cite{Dai} where, under the same conditions as in \cite{Wang2018} on the exponents $p,q\in\mathbb R$ (that is, $2<p<q<p(1+\frac{2}{n})$), explicit estimates for the blow-up time are provided.

Finally, we refer to \cite{Ball,Cao,Chen2015,Chen15,Dai,Di,Gopala,Yin,Khomrutai,Le2017,Peng,Zhu} for the study of the asymptotic behavior of weak solutions to \eqref{eq:main_eqConst} under high initial energy level conditions.

As for the case of variable exponents $p(.)$ and $q(.)$, to the best of our knowledge, there are no results in the existing literature in the spirit of those that we have just recalled. As a matter of fact, many of the techniques employed in the previous references to deal with \eqref{eq:main_eqConst} become unsuccessful to analyze \eqref{eq:main_eq} when $p(.)$ and $q(.)$ are measurable functions on $\overline{\Omega}$.

The goal of this paper is precisely to address this variable measurable coefficients framework. In particular, we shall identify sufficient conditions on $p(.)$, $q(.)$ and the initial datum $u_0$ for which the blow-up and the non-blow-up phenomena of solutions to \eqref{eq:main_eq} appear, also providing estimates for the blow-up time.

At this regard, we shall mention that this kind of results are already available for the hyperbolic version of \eqref{eq:main_eq}. The interested reader may refer for instance to \cite{Aboulaich,Amorim2013,Antontsev1,Antontsev2011,Antontsev7,Antontsev5,ChenLevine,Lian,Rahmoune2018,Rahmoune2019,SUN2016}. In particular, in \cite{Amorim2013,Antontsev2011}, the authors have discussed the Dirichlet problem for following equation
\begin{align*}
	&u_{tt}=\func{div}\left(a(x,t)|\nabla u|^{p(x,t)-2}\nabla u\right) +\alpha \Delta u_t+b(x,t)|u|^{\sigma(x,t)-2}u + f(x,t), \quad (x,t)\in\Omega\times(0,+\infty)
\end{align*}
with negative initial energy and, under suitable conditions on the functions $a,b,f,p,\sigma$, they have used Galerkin and energy methods to establish local existence, global existence and blow-up of the solutions.

The present paper is organized as follows: in Section \ref{sec:2}, we will introduce some preliminary concepts and notations that will be of use in our further analysis. Secondly, we will study the finite-time blow-up of solutions in Section \ref{sec:3}. The global existence of weak solutions to \eqref{eq:main_eq} is addressed in Section \ref{sec:4}, where we also we also establish decay estimates in the $H_0^1(\Omega)$-norm as $t\to+\infty$. Finally, in Section \ref{sec:5}, this decay result is illustrated by a two-dimensional numerical example.

\section{Preliminaries}\label{sec:2}

In this section, we present some preliminary concepts and notations that we shall employ in our further analysis. Let us start by introducing the variable-order Lebesque space $L^{p(.)}(\Omega)$, which is defined for all $p:\Omega \rightarrow [1,+\infty]$ a measurable function as
\begin{equation*}
	L^{p(.)}(\Omega):=\left\{u:\Omega\to\mathbb R \mbox{ measurable }: \int_{\Omega}\left\vert u(x)\right\vert^{p(x)}\,\mathrm{d}x<+\infty\right\}.
\end{equation*}
We then know that $L^{p(.)}(\Omega)$ is a Banach space, equipped with the Luxemburg-type norm
\begin{equation*}
	\left\Vert u\right\Vert_{p(.)}:= \inf \left\{\lambda >0\text{, \ } \int_{\Omega}\left\vert \frac{u(x)}{\lambda}\right\vert^{p(x)}\mathrm{d}x\leq 1\right\}.
\end{equation*}
Next, we define the variable-order Sobolev space $W^{1,p(.)}(\Omega)$ as
\begin{align*}
	W^{1,p(.)}(\Omega):= \left\{u\in L^{p(.)}(\Omega)\;:\; \nabla u \in L^{p(.)}(\Omega) \right\},
\end{align*}
equipped with the norm
\begin{align}\label{eq:SobolevNorm}
	\|u\|_{W^{1,p(.)}(\Omega)} = \|u\|_{p(.)}^2 + \|\nabla u\|_{p(.)}^2.
\end{align}
Moreover, in what follows we will need the following embedding result from \cite{Diening3,Fan}.
\begin{lemma}\label{lem:embedding}
Let $\Omega\subset\mathbb R^n$ be a bounded regular domain. It holds the following.
\begin{itemize}
	\item[1.] If $p\in C(\overline{\Omega})$ and $q:\Omega \rightarrow [1,+\infty)$ is a measurable function such that
	\begin{equation*}
		\underset{x\in\Omega}{\mbox{ess inf}}\,\big(p^{\ast}(x)-q(x)\big)>0,
	\end{equation*}
	with $p^{\ast}$ defined as in \eqref{eq:pAst}, then $W_0^{1,p(.)}(\Omega)\hookrightarrow L^{q(.)}(\Omega)$ with continuous and compact embedding.
	\item[2.] If $p$ satisfy \eqref{03}, then $\Vert u\Vert_{p(.)}\leq C\Vert \nabla u\Vert_{p(.)}$ for all $u\in W_0^{1,p(.)}(\Omega)$. In particular, $\Vert u\Vert_{1,p(.)}=\Vert \nabla u\Vert_{p(.)}$ defines a norm on $W_0^{1,p(.)}(\Omega)$ which is equivalent to \eqref{eq:SobolevNorm}.
\end{itemize}
\end{lemma}

Let us now introduce our notion of weak solution to \eqref{eq:main_eq}, and give a first result of local-in-time existence and uniqueness.
\begin{definition}
Let $u_0\in W_0^{1,p(.)}(\Omega)$ and $T>0$. A function
\begin{equation*}
	u\in L^{\infty}\left(0,T;W_0^{1,p(.)}(\Omega)\right) \text{ with } u_t\in L^{2}\left(0,T;H_0^1(\Omega)\right),
\end{equation*}
is called a weak solution to \eqref{eq:main_eq} if the identity
\begin{equation*}
	\left\langle u_t,v\right\rangle +\left\langle \nabla u_t,\nabla v\right\rangle +\left\langle |\nabla u|^{p(.)-2}\nabla u,\nabla v\right\rangle = \int_{\Omega}|u|^{q(.)-2}u\ln (|u|)v\,\mathrm{d}x
\end{equation*}
holds for any $v\in W_0^{1,p(.)}(\Omega)$, and for a.e. $t\in [0,T]$.
\end{definition}

\begin{theorem}
For all $u_0\in W_0^{1,p(.)}(\Omega)$, there exists $T_0>0$ such that the problem \eqref{eq:main_eq} admits a unique weak solution $u$ on $[0,T_0]$.
\end{theorem}

\begin{proof}
The result can be easily obtained by using the Faedo-Galerkin approach and combining the ideas from \cite{Akagi,Diening1,Fu,HuaWang} with the ones from \cite{Antonsev20} (see also \cite{YANGCAO}). We leave the details to the reader.
\end{proof}

\noindent Let now $\sigma>0$ be a positive constant satisfying
\begin{equation}\label{eq:sigma}
	0<\sigma < \begin{cases} p^\ast-q_2, & \text{ if } p(.) <n, \\ +\infty, & \text{ if }p(.) \geq n. \end{cases}
\end{equation}

Then, from Lemma \ref{lem:embedding} we have that $W_0^{1,q(.)+\sigma}(\Omega)\hookrightarrow L^{q(.)+\sigma}(\Omega)$, and there exists a positive constant $B_{\sigma}>0$ such that
\begin{equation}\label{eq:Bsigma}
	\Vert u\Vert_{q(.) +\sigma}\leq B_{\sigma}\Vert \nabla u\Vert_{q(.)+\sigma}.
\end{equation}
Finally, let us indicate with $\alpha_1$, $B_1$ and $E_1$ the following positive constants:
\begin{subequations}
	\begin{align}
		& \alpha_1 = \left(\frac{e\sigma q_1}{q_1+\sigma}B_1^{-\left(q_1+\sigma\right)}\right)^{\frac{p_2}{q_1-p_2+\sigma}}\label{eq:alpha1}
		\\[7pt]
		& B_1 = \max \left(1,B_{\sigma }\right)\label{eq:B1}
		\\[7pt]
		& E_1=\left(\frac{1}{p_2}-\frac{1}{q_1+\sigma}\right)\alpha_1.\label{eq:E1}
	\end{align}
\end{subequations}

\section{Finite-time blow-up of solutions}\label{sec:3}

In this section, we obtain a blow-up criterion for the solutions to \eqref{eq:main_eq}, also providing an upper bound for the blow-up time. To this end, let us introduce the energy associated with the solution of our problem \eqref{eq:main_eq}, which is defined as follows:
\begin{align}\label{eq:energy}
	E(t):= \int_{\Omega}\frac{1}{p(.)}\left\vert\nabla u \right\vert^{p(.)}\,\mathrm{d}x - \int_{\Omega}\frac{1}{q(.)}\left\vert u\right\vert^{q(.)}\ln (|u|)\,\mathrm{d}x +\int_{\Omega}\frac{1}{q^2(.)}|u|^{q(.)}\,\mathrm{d}x
\end{align}
We have the following result providing upper and lower bounds for $E(t)$.

\begin{lemma}\label{lem:energy_est}
Let $u_0\in W^{1,p(.)}(\Omega)$ and $E(t)$ be the energy associated with the corresponding solution of \eqref{eq:main_eq}, as defined in \eqref{eq:energy}. Let $g:[0,+\infty)\to\mathbb{R}$ be given by
\begin{align}\label{eq:g}
	g(\xi):=\frac{1}{p_2}\min \left(\xi^{\frac{p_1}{p_2}},\xi \right) -\frac{1}{e\sigma q_1}\max \left(B_1^{q_2+\sigma}\xi^{\frac{q_2+\sigma}{p_2}},B_1^{q_1+\sigma}\xi^{\frac{q_1+\sigma}{p_2}}\right),
\end{align}
where $\sigma$ and $B_1$ are the constants defined in \eqref{eq:sigma} and \eqref{eq:B1}, respectively. Let $\alpha:[0,+\infty)\to [0,+\infty)$ be defined as
\begin{align}\label{eq:alpha}
	\alpha(t) =\left\Vert \nabla u(t)\right\Vert_{p(.)}^{p_2}.
\end{align}
Then, we have
\begin{align}\label{eq:energy_est}
	g(\alpha)\leq E(t)\leq E(0).
\end{align}
\end{lemma}

\begin{proof}
First of all, a direct computation yields that
\begin{align}\label{eq:energyDer}
	\frac{dE(t)}{dt}=-\int_{\Omega}u_t^2\,\mathrm{d}x-\int_{\Omega}\left\vert \nabla u_t\right\vert^2 \,\mathrm{d}x\leq 0.
\end{align}

Hence, $E$ is decreasing with respect to $t$, and the inequality $E(t)\leq E(0)$ immediately holds for all $t\geq 0$. As for the lower estimate in \eqref{eq:energy_est}, we have that
\begin{align}\label{eq:energy_est_part}
	\int_{\Omega}\frac{1}{q(.)}\left\vert u \right\vert^{q(.)}\ln(|u|)\,\mathrm{d}x \leq \frac{1}{q_1}\int_\Omega\left\vert u\right\vert^{q(.)}\ln(|u|)\,\mathrm{d}x \leq \frac{1}{e\sigma q_1}\int_{\Omega}\left\vert u\right\vert^{q(.) +\sigma}\,\mathrm{d}x,
\end{align}
where we used that
\begin{align}\label{eq:estLog}
	\ln(\xi)\leq\frac{1}{e\sigma}\xi^{\sigma}, \quad\mbox{ for all } \xi, \sigma>0.
\end{align}
Therefore, we get from \eqref{eq:energy}, \eqref{eq:alpha}, \eqref{eq:g} and \eqref{eq:energy_est_part} that
\begin{align}\label{eq:est_g_energy}
	E(t) \geq &\, \frac{1}{p_2}\min \left(\left\Vert \nabla u\right\Vert_{p(.)}^{p_1},\left\Vert\nabla u\right\Vert_{p(.)}^{p_2}\right)-\frac{1}{e\sigma
		q_1}\max \left(\left\Vert u\right\Vert_{q(.)+\sigma}^{q_2+\sigma},\left\Vert u\right\Vert_{q(.) +\sigma}^{q_1+\sigma}\right) 
	\\
	\geq &\, \frac{1}{p_2}\min \left(\left\Vert\nabla u\right\Vert_{p(.)}^{p_1},\left\Vert \nabla u\right\Vert_{p(.)}^{p_2}\right) -\frac{1}{e\sigma q_1}\max \left(\left(B_1\left\Vert \nabla u\right\Vert_{p(.)}\right)^{q_2+\sigma},\left(B_1\left\Vert \nabla u\right\Vert_{p(.)}\right)^{q_1+\sigma}\right) = g(\alpha). \notag
\end{align}
This concludes the proof of \eqref{eq:energy_est}.
\end{proof}
\noindent Next, we state some technical lemmas which will be needed in the proof of our main result.
\begin{lemma}\label{lem:lemma_h}
Let $\alpha_1$, $B_1$ and $E_1$ be given as in \eqref{eq:alpha1}, \eqref{eq:B1} and \eqref{eq:E1}, respectively. Define the function $h:[0,+\infty)\rightarrow \mathbb{R}$ as
\begin{equation}\label{eq:h}
	h(\xi) =\frac{1}{p_2}\xi -\frac{1}{e\sigma q_1}B_1^{q_1+\sigma}\xi^{\frac{q_1+\sigma}{p_2}}.
\end{equation}
Then, the following assertions hold:
\begin{itemize}
	\item[1.] $h$ is increasing for $0<\xi \leq \alpha_1$ and decreasing for $\xi \geq \alpha_1$.
		
	\item[2.] $h(\alpha_1) = E_1$.
	
	\item[3.] $\displaystyle\lim_{\xi\rightarrow +\infty}h(\xi) =-\infty$.
\end{itemize}
\end{lemma}

\begin{proof}
First of all, from the definition \eqref{eq:h} of the function $h$ we can easily compute
\begin{equation*}
	h^{\prime}(\xi)=\frac{1}{p_2}-\frac{q_1+\sigma}{p_2}\frac{1}{e\sigma q_1}B_1^{q_1+\sigma}\xi^{\frac{q_1+\sigma}{p_2}-1}.
\end{equation*}

Then, it is straightforward to check that $h'(\xi)>0$ for $\xi<\alpha_1$ and $h'(\xi)<0$ for $\xi>\alpha_1$. Besides, the fact that $h(\alpha_1) = E_1$ follows directly from \eqref{eq:E1} and \eqref{eq:h}. Finally, since by definition of $p_2$, $q_1$ and $\sigma$ we have $q_1-p_2+\sigma >0$, we get
\begin{align*}
	\lim_{\xi\to +\infty} h(\xi) = \lim_{\xi\to +\infty} \xi\left(\frac{1}{p_2} - \frac{1}{e\sigma q_1}B_1^{q_1+\sigma}\xi^{\frac{q_1-p_2+\sigma}{p_2}}\right) = -\infty.
\end{align*}
This concludes our proof.
\end{proof}

\begin{lemma}\label{lem:alpha2}
Let $\sigma$, $\alpha_1$, $B_1$ and $E_1$ be given as in \eqref{eq:sigma}, \eqref{eq:alpha1}, \eqref{eq:B1} and \eqref{eq:E1}, respectively. Assume that the initial value $u_0$ is chosen so that
\begin{equation}\label{eq:u0hyp}
	0\leq E(0)<\frac{p_2}{q_1+p_2} \left(E_1-\frac{\sigma\alpha_1}{p_2(q_1+\sigma)}\right)\quad\mbox{ and }\quad \alpha_1< \left\Vert \nabla u_0\right\Vert_{p(.)}^{p_2}\leq B_1^{-p_2}.
\end{equation}
Then ,there exists a positive constant $\alpha_2>\alpha_1$ such that
\begin{subequations}
	\begin{align}
		& \left\Vert \nabla u \right\Vert_{p(.)}^{p_2}\geq \alpha_2, \quad\mbox{ for all } t\geq 0	\label{eq:alpha2_1}
		\\[10pt]
		&\int_{\Omega}\frac{1}{q(.)}\left\vert u\right\vert^{q(.)}\ln(|u|)\mathrm{d}x\geq \frac{1}{e\sigma q_1}B_1^{q_1+\sigma}\alpha_2^{\frac{q_1+\sigma}{p_2}} \label{eq:alpha2_2}
		\\[10pt]
		&\frac{\alpha_2}{\alpha_1}\geq \left(\left(q_1+\sigma\right)\left(\frac{1}{p_2}-\frac{E(0)}{\alpha_1}\right)\right)^{\frac{p_2}{q_1+\sigma-p_2}}>1. \label{eq:alpha2_3}
	\end{align}
\end{subequations}
\end{lemma}

\begin{proof}
Since from \eqref{eq:u0hyp} we have
\begin{align*}
	0\leq E(0) <\frac{p_2E_1}{q_1+p_2}<E_1,	
\end{align*}
it follows from Lemma \ref{lem:lemma_h} that there exists a positive constant $\alpha_2>\alpha_1$ such that the initial energy $E(0)$ satisfies  $E(0) =h(\alpha_2)$ (see Figure \ref{fig:h}). Moreover, since $B_1 >1$ (see \eqref{eq:B1}), it follows again from \eqref{eq:u0hyp} that
\begin{align*}
	\alpha_0 = \left\Vert \nabla u_0\right\Vert_{p(.)}^{p_2}\leq B_1^{-p_2} <1.
\end{align*}

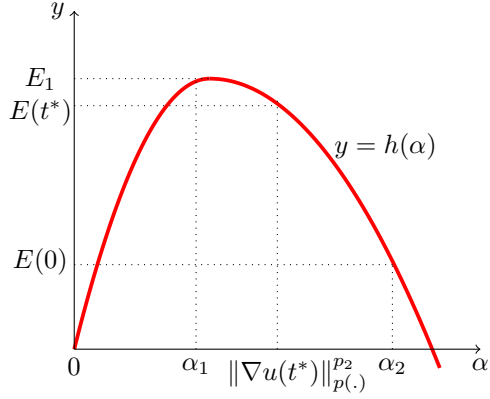
\begin{SCfigure}[0.6][h]
\begin{tikzpicture}[scale=0.9]	
	\draw[red, line width = 0.50mm]   plot[smooth,domain=0:3.4] (\x, {4-0.37*\x^2});
	\draw[red, line width = 0.50mm]   plot[smooth,domain=-2:0] (\x, {4+\x^2});
	\draw[->] (-2,0) -- (4,0) node[anchor=north] {$\alpha$};
	\draw	(-2,0) node[anchor=north] {0}
	(-.2,0) node[anchor=north] {$\alpha_{1}$}(1.3,0) node[anchor=north] {$\left\Vert \nabla u(t^{\ast})\right\Vert_{p(.)}^{p_2}$ };
	\draw[dotted] (1,0) -- (1,3.7);
	\draw[dotted] (-2,1.25) -- (2.7,1.25);
	\draw	(-2.5,1.3)  node{$E(0)$};
	\draw[dotted] (-2,3.6) -- (1,3.6);
	\draw (-2.5,3.5)  node{$E(t^{*})$};
	\draw (2.7,0) node[anchor=north] {$\alpha_{2}$};
	\draw[dotted] (2.7,0) -- (2.7,1.3);
	\draw[dotted] (-2,4) -- (0,4);
	\draw[->] (-2,0) -- (-2,5) node[anchor=east] {$y$};
	\draw[dotted] (-0.2,0) -- (-0.2,4);
	\draw (-2.5,4)  node{$E_{1}$};
	\draw (2.6,3) node {$y=h(\alpha)$}; 
\end{tikzpicture}
\vspace{-0.5cm}\caption{Diagram for the proof of Lemma \protect\ref{lem:alpha2}. We plot the function $h$, highlighting the quantities $\alpha_1$, $\alpha_2$ and $\|\nabla u(t^\ast)\|_{p(.)}^{p_2}$ with the corresponding values of the energy ($E_1$, $E(0)$ and $E(t^\ast)$, respectively).}\label{fig:h}
\end{SCfigure}

Therefore, recalling the definitions \eqref{eq:g} and \eqref{eq:h} of the functions $g$ and $h$, since $p_1<p_2$, and using \eqref{eq:est_g_energy}, we can conclude that
\begin{align*}
	h(\alpha_0) = g(\alpha_0)\leq E(0) = h(\alpha_2).
\end{align*}
Hence, it follows from Point 1 in Lemma \ref{lem:lemma_h} that $\alpha_0\geq\alpha_2$, and \eqref{eq:alpha2_1} holds for $t=0$.

To prove \eqref{eq:alpha2_1} for all $t\geq 0$, we argue by contradiction. Suppose that there exists some strictly positive time $t^{\ast}>0$ such that $\left\Vert \nabla u(t^{\ast})\right\Vert_{p(.)}^{p_2}<\alpha_2$. Then, by the continuity of the $L^{p(.)}$-norm, and since $\alpha_2>\alpha_1$, we may take $t^{\ast}$ such that
\begin{align*}
	\alpha_2>\left\Vert \nabla u(t^{\ast})\right\Vert_{p(.)}^{p_2}>\alpha_1.
\end{align*}
Hence, it follows from Lemma \ref{lem:lemma_h} and \eqref{eq:est_g_energy} that
\begin{equation*}
	E(0) =h(\alpha_2)<h\left(\left\Vert \nabla u(t^{\ast})\right\Vert_{p(.)}^{p_2}\right) \leq E(t^\ast).
\end{equation*}

This contradicts the fact that the energy $E(t)$ is decreasing, as we proved in Lemma \ref{lem:energy_est}. Therefore, \eqref{eq:alpha2_1} holds.

Let us now prove \eqref{eq:alpha2_2}. At this regard, from the definition of the energy \eqref{eq:energy}, and since $E(t)\leq E(0)$, we obtain
\begin{align}\label{eq:int_est_prel}
	\int_{\Omega}\frac{1}{q(.)}\left\vert u\right\vert^{q(.)}\ln(|u|)\mathrm{d}x &= \int_{\Omega}\frac{1}{p(.)}\left\vert\nabla u \right\vert^{p(.)}\,\mathrm{d}x +\int_{\Omega}\frac{1}{q^2(.)}|u|^{q(.)}\,\mathrm{d}x - E(t) \notag
	\\
	&\geq \int_{\Omega}\frac{1}{p(.)}\left\vert\nabla u \right\vert^{p(.)}\,\mathrm{d}x - E(0) \notag 
	\\
	&\geq \frac{1}{p_2}\min \left(\left\Vert\nabla u\right\Vert_{p(.)}^{p_2},\left\Vert \nabla u\right\Vert_{p(.)}^{p_1}\right) -E(0) \notag
	\\
	& \geq \frac{1}{p_2}\min \left(\alpha_2^{\frac{p_1}{p_2}},\alpha_2\right)-E(0).
\end{align}
Moreover, since $\alpha_2\leq\alpha_0<1$ and $p_1\leq p_2$, we have
\begin{align*}
	\min \left(\alpha_2^{\frac{p_1}{p_2}},\alpha_2\right) = \alpha_2.
\end{align*}
Hence, it follows from \eqref{eq:h} and \eqref{eq:int_est_prel} that
\begin{align*}
	\int_{\Omega}\frac{1}{q(.)}\left\vert u \right\vert^{q(.)}\ln(|u|)\mathrm{d}x \geq \frac{1}{p_2}\alpha_2-E(0) = \frac{1}{p_2}\alpha_2-h(\alpha_2) =\frac{1}{e\sigma q_1}B_1^{q_1+\sigma}\alpha_2^{\frac{q_1+\sigma}{p_2}}.
\end{align*}

Finally, let us prove \eqref{eq:alpha2_3}. To this end, since $E(0)<E_1$, recalling the definition of $E_1$ given in \eqref{eq:E1} we can readily check that
\begin{equation*}
	\left(\left(q_1+\sigma\right)\left(\frac{1}{p_2}-\frac{E(0)}{\alpha_1}\right)\right)^{\frac{p_2}{q_1-p_2+\sigma}} > \left(\left(q_1+\sigma\right)\left(\frac{1}{p_2}-\frac{E_1}{\alpha_1}\right)\right)^{\frac{p_2}{q_1-p_2+\sigma}} = 1.
\end{equation*}
Hence, the second inequality in \eqref{eq:alpha2_3} holds. As for the first inequality, we can easily compute
\begin{align*}
	E(0) &= h(\alpha_2) = \alpha_2\left(\frac{1}{p_2} -\frac{1}{e\sigma q_1}B_1^{q_1+\sigma}\alpha_2^{\frac{q_1+\sigma}{p_2}-1}\right) = \alpha_1\frac{\alpha_2}{\alpha_1}\left(\frac{1}{p_2} -\frac{1}{e\sigma q_1}B_1^{q_1+\sigma}\left(\frac{\alpha_2}{\alpha_1}\right)^{\frac{q_1+\sigma}{p_2}-1}\alpha_1^{\frac{q_1+\sigma}{p_2}-1}\right).
\end{align*}
Recalling the definition of $\alpha_1$ given in \eqref{eq:alpha1}, and since $\alpha_2>\alpha_1$ we then have that
\begin{align*}
	E(0) = \alpha_1\frac{\alpha_2}{\alpha_1}\left(\frac{1}{p_2}-\frac{1}{q_1+\sigma}\left(\frac{\alpha_2}{\alpha_1}\right)^{\frac{q_1+\sigma-p_2}{p_2}}\right) \geq \alpha_1\left(\frac{1}{p_2}-\frac{1}{q_1+\sigma}\left(\frac{\alpha_2}{\alpha_1}\right)^{\frac{q_1+\sigma -p_2}{p_2}}\right).
\end{align*}
Therefore,
\begin{align*}
	\frac{E(0)}{\alpha_1} \geq \frac{1}{p_2}-\frac{1}{q_1+\sigma}\left(\frac{\alpha_2}{\alpha_1}\right)^{\frac{q_1+\sigma -p_2}{p_2}},
\end{align*}
from which it follows immediately that
\begin{align*}
	\frac{\alpha_2}{\alpha_1} \geq \left((q_1+\sigma)\left(\frac{1}{p_2}-\frac{E(0)}{\alpha_1}\right)\right)^{\frac{p_2}{q_1+\sigma -p_2}}
\end{align*}
Our proof is then concluded.
\end{proof}

\noindent Let us define
\begin{equation}\label{eq:H}
	H(t) := E_1 - \frac{\sigma\alpha_1}{p_2(q_1+\sigma)} - E(t), \quad \text{ for }t\geq 0.
\end{equation}
We have the following result.

\begin{lemma}\label{lem:H}
Let $\sigma$, $\alpha_1$, $B_1$ and $E_1$ be given as in \eqref{eq:sigma}, \eqref{eq:alpha1}, \eqref{eq:B1} and \eqref{eq:E1}, respectively. Assume that the initial value $u_0$ is chosen so that
\begin{equation*}
	0\leq E(0)<\frac{p_2}{q_1+p_2}\left(E_1-\frac{\sigma\alpha_1}{p_2(q_1+\sigma)}\right)\quad\mbox{ and }\quad \alpha_1< \left\Vert \nabla u_0\right\Vert_{p(.)}^{p_2}\leq B_1^{-p_2}.
\end{equation*}
Then, the functional $H(t)$ defined in \eqref{eq:H} satisfies the following estimates:
\begin{equation}\label{eq:est_H}
	0<H(0)\leq H(t)\leq\int_{\Omega}\frac{1}{q(.)}\left\vert u \right\vert^{q(.)}\ln(|u|)\mathrm{d}x,\quad t\geq 0.
\end{equation}
\end{lemma}

\begin{proof}
By Lemma \ref{lem:energy_est}, we know that $H(t)$ is non-decreasing for all $t\geq 0$. Hence
\begin{equation}\label{eq:est_H_prel}
	H(t)\geq H(0) = E_1 - \frac{\sigma\alpha_1}{p_2(q_1+\sigma)} - E(0).
\end{equation}
By the definition \eqref{eq:energy} of $E(t)$, we have
\begin{align*}
	H(t) & -\int_{\Omega}\frac{1}{q(.)}\left\vert u \right\vert^{q(.)}\ln(|u|)\,\mathrm{d}x 
	\\
	&= E_1 - \frac{\sigma\alpha_1}{p_2(q_1+\sigma)} - \int_{\Omega}\frac{1}{p(.)}\left\vert \nabla u \right\vert^{p(.)}dx -\int_{\Omega}\frac{1}{q^2(.)}\left\vert u \right\vert^{q(.)}\,\mathrm{d}x
	\\
	& <E_1-\int_{\Omega}\frac{1}{p(.)}\left\vert \nabla u \right\vert^{p(.)}\,\mathrm{d}x \notag.
\end{align*}
Hence, using \eqref{eq:alpha1}, \eqref{eq:E1}, \eqref{eq:alpha2_1} and the fact that $\alpha_2>\alpha_1$, for all $t\geq 0$ we have
\begin{align}\label{eq:est_H_prel2}
	H(t) -\int_{\Omega}\frac{1}{q(.)}\left\vert u \right\vert^{q(.)}\ln(|u|)\mathrm{d}x & < E_1-\frac{1}{p_2}\min \left(\left\Vert \nabla u\right\Vert_{p(.)}^{p_2},\left\Vert \nabla u \right\Vert_{p(.)}^{p_1}\right) 
	\\
	&<\left(\frac{1}{p_2}-\frac{1}{q_1+1}\right) \alpha_1-\frac{1}{p_2}\alpha_1<0.\notag 
\end{align}
Thus, \eqref{eq:est_H} follows immediately from \eqref{eq:est_H_prel} and \eqref{eq:est_H_prel2}.
\end{proof}

We are now ready to state and prove our main result of this Section, concerning the blow-up of solutions for \eqref{eq:main_eq}.
\begin{theorem}\label{thm:Theo1}
Let $\sigma$, $\alpha_1$, $B_1$ and $E_1$ be given as in \eqref{eq:sigma}, \eqref{eq:alpha1}, \eqref{eq:B1} and \eqref{eq:E1}, respectively. Assume that the initial value $u_0$ is chosen so that \eqref{eq:u0hyp} holds. Then, the corresponding solution of problem \eqref{eq:main_eq} will blow-up in finite time $T^\ast$. Furthermore, we have the following upper estimate for $T^\ast$:
\begin{equation*}
	0< T^\ast < \frac{\Vert u_0\Vert_{H_0^1(\Omega)}^{p_2}}{(\mathcal B+p_2-2)\Big(\frac{q_1-p_2}{q_1+\sigma}\alpha_1 - (\mathcal B + p_2)E(0)\Big)},
\end{equation*}
where we have denoted
\begin{align}\label{eq:B}
	\mathcal B = \mathcal B(\alpha_1,\alpha_2,p_2,q_1,\sigma) := (q_1-p_2)\left(1-\left(\frac{\alpha_1}{\alpha_2}\right)^{\frac{\sigma+q_1}{p_2}}\right)>0,
\end{align}
with $\alpha_2$ as in Lemma \ref{lem:alpha2}.
\end{theorem}

\begin{proof}
Let us define the function
\begin{equation*}
	\varphi (t) =\frac{1}{2}\int_{\Omega}u^2\,\mathrm{d}x+\frac 12\int_{\Omega}\left\vert \nabla u\right\vert^2\,\mathrm{d}x= \frac 12\Vert u \Vert_{H_0^1(\Omega)}^2.
\end{equation*}
Then, using \eqref{eq:main_eq} and integration by parts, the derivative $\varphi^{\prime}(t)$ satisfies
\begin{align*}
	\varphi^{\prime}(t) &=\int_{\Omega}uu_t \,\mathrm{d}x + \int_{\Omega}\nabla u\cdot\nabla u_t \,\mathrm{d}x \notag
	\\
	&=\int_{\Omega}u\left(\Delta u_t+\operatorname{div}\left(\left\vert \nabla u\right\vert^{p(.)-2}\nabla u\right) +\left\vert u\right\vert^{q(.)-2}u\ln(|u|)\right)\,\mathrm{d}x - \int_{\Omega} u\Delta u_t \,\mathrm{d}x \notag
	\\
	&=-\int_{\Omega}\left\vert \nabla u \right\vert^{p(.)}\,\mathrm{d}x + \int_{\Omega}\left\vert u\right\vert^{q(.)}\ln (|u|)\,\mathrm{d}x.
\end{align*}

Recalling the definitions \eqref{eq:energy} and \eqref{eq:H} of the energy $E(t)$ and the function $H(t)$, from the above identity we can estimate
\begin{align}\label{eq:phiPrime_est}
	\varphi^{\prime}(t) &\geq -p_2E(t)-p_2\int_{\Omega}\frac{1}{q(.)}\left\vert u \right\vert^{q(.)}\ln(|u|)\mathrm{d}x +\int_{\Omega}\left\vert  u\right\vert^{q(.)}\ln(|u|)\,\mathrm{d}x \notag
	\\
	&\geq -\left(p_2E_1 - \frac{\sigma\alpha_1}{q_1+\sigma}\right) + p_2H(t)+\frac{q_1-p_2}{q_1}\int_{\Omega}\left\vert u\right\vert^{q(.)}\ln(|u|)\,\mathrm{d}x.
\end{align}
Moreover, from \eqref{eq:alpha1}, \eqref{eq:E1} and \eqref{eq:alpha2_2} we have
\begin{align}\label{eq:p2E1}
	p_2E_1 - \frac{\sigma\alpha_1}{q_1+\sigma} & =\frac{q_1-p_2}{q_1+\sigma}\alpha_1 
	\\
	&=\frac{q_1-p_2}{q_1+\sigma}\alpha_1^{\frac{q_1+\sigma}{p_2}} \alpha_1^{-\frac{q_1-p_2+\sigma}{p_2}} = \frac{q_1-p_2}{q_1+\sigma}\alpha_1^{\frac{q_1+\sigma}{p_2}}\left(\frac{e\sigma q_1}{q_1+\sigma}B_1^{-(q_1+\sigma)}\right)^{-1} \notag 
	\\
	&=\left(\frac{e\sigma q_1}{q_1+\sigma}\right)^{-1}\frac{q_1-p_2}{q_1+\sigma}B_1^{q_1+\sigma}\alpha_1^{\frac{q_1+\sigma}{p_2}} =\frac{q_1-p_2}{q_1}\left(\frac{1}{e\sigma}B_1^{q_1+\sigma}\alpha_1^{\frac{q_1+\sigma}{p_2}}\right)\notag
	\\
	&= \frac{q_1-p_2}{q_1}\left(\frac{\alpha_1}{\alpha_2}\right)^{\frac{q_1+\sigma}{p_2}}\left(\frac{1}{e\sigma q_1}B_1^{q_1+\sigma}\alpha_2^{\frac{q_1+\sigma}{p_2}}\right) \notag 
	\\
	&\leq \frac{q_1-p_2}{q_1}\left(\frac{\alpha_1}{\alpha_2}\right)^{\frac{q_1+\sigma}{p_2}}\int_{\Omega}\left\vert u\right\vert^{q(.)}\ln(|u|)\,\mathrm{d}x. \notag
\end{align}
Then, it follows \eqref{eq:phiPrime_est} and \eqref{eq:p2E1} that
\begin{equation}\label{eq:phiDer}
	\varphi^{\prime}(t) \geq \frac{q_1-p_2}{q_1}\left(1-\left(\frac{\alpha_1}{\alpha_2}\right)^{\frac{q_1+\sigma}{p_2}}\right) \int_{\Omega}\left\vert u\right\vert^{q(.)}\ln(|u|)\,\mathrm{d}x + p_2H(t).
\end{equation}
Besides, using \eqref{eq:sigma}, we can readily check that
\begin{equation}\label{eq:zeta}
	\zeta:= \frac{q_1-p_2}{q_1}\left(1-\left(\frac{\alpha_1}{\alpha_2}\right)^{\frac{q_1+\sigma}{p_2}}\right)>0.
\end{equation}
Now, let us define the function
\begin{align}\label{eq:psi}
	\psi(t) := -\zeta q_1E(t) + p_2H(t).
\end{align}
Then, using \eqref{eq:energyDer} and \eqref{eq:H}, we have
\begin{align}\label{eq:psiDer}
	\psi^{\prime}(t) & = -\zeta q_1 E'(t) + p_2H'(t) 
	\\
	&= -(\zeta q_1+p_2)E'(t) = (\zeta q_1+p_2) \int_{\Omega}\left(\left\vert u_t\right\vert^2+\left\vert \nabla u_t\right\vert^2\right)\, \mathrm{d}x>0. \notag
\end{align}
Hence, using Cauchy-Schwarz's inequality, we obtain
\begin{align}\label{eq:phipsiDer}
	\varphi(t) \psi^{\prime}(t) & =\frac{\zeta q_1+p_2}{2}\left[\int_{\Omega}\left(\left\vert u\right\vert^2+\left\vert \nabla u\right\vert^2\right) \mathrm{d}x\right] \left[\int_{\Omega}\left(\left\vert u_t\right\vert^2+\left\vert \nabla u_t\right\vert^2\right) \mathrm{d}x\right] \notag
	\\
	& \geq \frac{\zeta q_1+p_2}{2}\left(\int_{\Omega}uu_t \mathrm{d}x+\int_{\Omega}\nabla u\nabla u_t \mathrm{d}x\right)^2 \notag
	\\
	& =\frac{\zeta q_1+p_2}{2}\varphi^{\prime}(t)^2.
\end{align}
Since, by definition, $\zeta<1$, we have from \eqref{eq:u0hyp} that
\begin{align*}
	E(0)<\frac{p_2}{q_1+p_2}\left(E_1-\frac{\sigma\alpha_1}{p_2(q_1+\sigma)}\right)<\frac{p_2}{\zeta q_1+p_2}\left(E_1-\frac{\sigma\alpha_1}{p_2(q_1+\sigma)}\right).
\end{align*}
Hence, we obtain that
\begin{align}\label{eq:psiZero}
	\psi(0) &= -\zeta q_1E(0) + p_2H(0) 
	\\
	&= -(\zeta q_1+p_2)E(0) + p_2\left(E_1-\frac{\sigma\alpha_1}{p_2(q_1+\sigma)}\right) >0 \notag
\end{align}
which, thanks to \eqref{eq:psiDer}, yields $\psi(t) >0$ for all $t\geq 0$. Moreover, from \eqref{eq:phiDer} and \eqref{eq:psi} it is easy to check that $\varphi^{\prime }(t) \geq \psi(t)$. Hence from \eqref{eq:phipsiDer}, we get
\begin{equation*}
	\varphi(t)\psi^{\prime }(t) \geq \frac{\zeta q_1+p_2}{2}\varphi^{\prime }(t)\psi(t)
\end{equation*}
which can be written as
\begin{equation}\label{eq:phipsiDerEst}
	\frac{\psi^{\prime }(t)}{\psi(t)}\geq \frac{\zeta q_1+p_2}{2}\frac{\varphi^{\prime }(t)}{\varphi(t)}.
\end{equation}
Integrating \eqref{eq:phipsiDerEst} from $0$ to $t$ and using \eqref{eq:psi}, we have
\begin{equation}\label{eq:phipsiDerEst2}
	\frac{\varphi^{\prime}(t)}{\varphi(t)^{\frac{\zeta q_1+p_2}{2}}}\geq \frac{\psi(0)}{\varphi(0)^{\frac{\zeta q_1+p_2}{2}}}.
\end{equation}%
Integrating also \eqref{eq:phipsiDerEst2} from $0$ to $t$, we obtain
\begin{align*}
	-\frac{2}{\zeta q_1+p_2-2}\left. \varphi(s)^{-\frac{\zeta q_1+p_2}{2}-1}\right|_{s=0}^{s=t} \geq \frac{\psi(0)}{\varphi(0)^{\frac{\zeta q_1+p_2}{2}}} t,
\end{align*}
that is,
\begin{align*}
	\frac{1}{\varphi(t)^{\frac{\zeta q_1+p_2-2}{2}}} \leq \frac{1}{\varphi(0)^{\frac{\zeta q_1+p_2-2}{2}}} - \frac{\zeta q_1+p_2-2}{2}\frac{\psi(0)}{\varphi(0)^{\frac{\zeta q_1+p_2}{2}}} t,
\end{align*}

From the above inequality, by means of some easy algebraic manipulations, we finally get
\begin{align*}
	\varphi(t)^{\frac{\zeta q_1+p_2-2}{2}} \geq \frac{\varphi(0)^{\frac{\zeta q_1+p_2}{2}}}{\varphi(0) - \frac{\zeta q_1+p_2-2}{2}\psi(0)t}
\end{align*}
Let
\begin{equation}\label{eq:Tast}
	0<T^{\ast }=\frac{2\varphi(0)}{(\zeta q_1+p_2-2)\psi(0)}.
\end{equation}
Then
\begin{align*}
	\varphi(0) - \frac{\zeta q_1+p_2-2}{2}\psi(0)T^\ast = 0,
\end{align*}
yielding the blowing-up of $\varphi (t)$ at time $T^{\ast}$. By definition of $\varphi$, we can then conclude that $u(x,t)$ blows-up in $H_0^1(\Omega)$-norm at $t=T^\ast$.

Finally, notice that by means of \eqref{eq:B}, \eqref{eq:zeta} and \eqref{eq:psiZero}, and recalling the definition of $\varphi$, we can easily obtain from \eqref{eq:Tast}
\begin{align*}
	T^{\ast } = \frac{\|u_0\|^2_{H_0^1(\Omega)}}{(\mathcal B+p_2-2)\left(\frac{q_1-p_2}{q_1+\sigma}\alpha_1-(\mathcal B+p_2)E(0)\right)}.
\end{align*}
This concludes our proof.
\end{proof}

\section{Non-blow-up case}\label{sec:4}

In this section, we present some non blow-up conditions for the solution of \eqref{eq:main_eq}. In particular, we are going to show that, when the initial datum $u_0$ is small enough in $H_0^1(\Omega)$, blow-up cannot manifest at any time $t\geq 0$. In addition to that, we will also provide explicit time-decay estimates for the $H_0^1(\Omega)$-norm of the solution.

In order to state our main result in this section, let us denote with $\lambda_1$ the first eigenvalue of the Dirichlet Laplacian on $\Omega$ and introduce the following constants
\begin{equation}\label{eq:constantsAB}
	\mathrm{A}=\left\vert \Omega \right\vert^{\frac{2-p_1}{2}}\left(\frac{\lambda_1}{2(\lambda_1+1)}\right)^{\frac{p_1}{2}}, \quad \mathrm{B}=\frac{4}{e\sigma}B_{\sigma}\mathcal C_1^{q_1+\sigma}, \quad \mathrm{C} =\frac{4}{e\sigma}B_{\sigma}\mathcal C_2^{q_2+\sigma}.
\end{equation}

In \eqref{eq:constantsAB}, $\mathcal C_1,\mathcal C_2>0,$ are the best embedding constants of $H_0^1(\Omega)$ into $W_0^{1,q_1+\sigma}(\Omega)$ and $W_0^{1,q_2+\sigma}(\Omega)$, respectively. Moreover, $\sigma$ is defined as in \eqref{eq:sigma} and $B_\sigma$ is the constant introduced in \eqref{eq:Bsigma} for the embedding $W_0^{1,p(.)}(\Omega)\hookrightarrow L^{q(.) +\sigma}(\Omega)$. We then have the following non blow-up result for the weak solutions of \eqref{eq:main_eq}.

\begin{theorem}\label{thm:Theo3}
Let $u_0\in W_0^{1,p(.)}(\Omega)$ satisfy
\begin{align}\label{eq:u0cond}
	\Vert u_0\Vert_{H_0^1(\Omega)}<\min \left(\left(\frac{\mathrm{A}}{\mathrm{2B}}\right)^{\frac{1}{q_1-p_1+\sigma}},\left(\frac{\mathrm{A}}{\mathrm{2C}}\right)^{\frac{1}{q_2-p_1+\sigma}}\right),
\end{align}
where the constants $\mathrm{A}$, $\mathrm{B}$ and $\mathrm{C}$ have been defined in \eqref{eq:constantsAB}. Then, the corresponding weak solution $u$ of \eqref{eq:main_eq} cannot blow-up in $H_0^1(\Omega)$-norm, and the following decay estimates hold
\begin{subequations}
	\begin{align}
		&\Vert u(t)\Vert_{H_0^1(\Omega)}\leq \left(\frac{4}{\mathrm{A}(p_1-2)t+4\left\Vert u_0\right\Vert_{H_0^1(\Omega)}^{2-p_1}}\right)^{\frac{1}{p_1-2}}, \quad \mbox{ if } p_1>2 \label{eq:decay1}
		\\[10pt]
		&\Vert u(t)\Vert_{H_0^1(\Omega)}\leq \left(\frac{A\|u_0\|^{q_1-2+\sigma}_{H_0^1(\Omega)}}{\mathrm{A}-B\left\Vert u_0\right\Vert_{H_0^1(\Omega)}^{q_1-2+\sigma}}\right)^{\frac{1}{q_1-2+\sigma}}e^{-\frac A2 t}, \quad \mbox{ if } p_1=2. \label{eq:decay2}
	\end{align}
\end{subequations}
\end{theorem}

\begin{proof}
Define the auxiliary function
\begin{equation}\label{eq:theta}
	\theta(t):= 2\varphi(t) =\int_{\Omega}u^2\,\mathrm{d}x+\int_{\Omega}\left\vert \nabla u\right\vert^2\,\mathrm{d}x=\Vert u\Vert_{H_0^1(\Omega)}^2.
\end{equation}
We then have from \eqref{eq:main_eq} and integration by parts that
\begin{align}\label{eq:thetaDer}
	\theta^{\prime}(t) & = 2\int_{\Omega}uu_t\,\mathrm{d}x + 2\int_{\Omega}\nabla u\nabla u_t\,\mathrm{d}x = -2\int_{\Omega}\left\vert\nabla u\right\vert^{p(.)}\,\mathrm{d}x + 2 \int_\Omega |u|^{q(.)}\ln(|u|)\,\mathrm{d}x.
\end{align}
For any $t>0$, we split the domain $\Omega$ into $\Omega=\Omega_1\cup\Omega_2$ with
\begin{equation*}
	\Omega_1:= \Big\{x\in \Omega \,:\,|\nabla u(x,t)|<1\Big\}\quad\text{ and }\quad \Omega_2=\Big\{x\in \Omega \,:\,|\nabla u(x,t)|\geq 1\Big\}.
\end{equation*}
We then have
\begin{subequations}
	\begin{align}
		\int_{\Omega}\left\vert \nabla u\right\vert^{p(.)}\mathrm{d}x &= \int_{\Omega_1}\left\vert \nabla u\right\vert^{p(.)}\,\mathrm{d}x + \int_{\Omega_2}\left\vert \nabla u\right\vert^{p(.)}\mathrm{d}x \notag
		\\
		&\geq \int_{\Omega_1}\left\vert \nabla u\right\vert^{p_2}\,\mathrm{d}x + \int_{\Omega_2}\left\vert \nabla u\right\vert^{p_1}\mathrm{d}x\geq \int_{\Omega_2}\left\vert \nabla u\right\vert^{p_1}\mathrm{d}x.\label{eq:gradEst1}
		\\[7pt]
		\int_{\Omega}\left\vert \nabla u\right\vert^{\rho+\sigma}\mathrm{d}x &= \int_{\Omega_1}\left\vert \nabla u\right\vert^{\rho+\sigma}\mathrm{d}x + \int_{\Omega_2}\left\vert \nabla u\right\vert^{\rho+\sigma}\mathrm{d}x \leq 2\int_{\Omega_2}\left\vert \nabla u\right\vert^{\rho+\sigma}\mathrm{d}x.\label{eq:gradEst2}
	\end{align}
\end{subequations}
Using \eqref{eq:Bsigma}, \eqref{eq:estLog}, \eqref{eq:gradEst1} and \eqref{eq:gradEst2}, from \eqref{eq:thetaDer} we can then estimate
\begin{align}\label{eq:thetaDerEst}
	\theta^{\prime}(t) & \leq -2\int_{\Omega_2}\left\vert \nabla u\right\vert^{p_1}\,\mathrm{d}x + \frac{2}{e\sigma}\int_{\Omega}\left\vert u\right\vert^{q(.)+\sigma}\,\mathrm{d}x
	\\
	& \leq -2\int_{\Omega_2}\left\vert \nabla u\right\vert^{p_1}\,\mathrm{d}x + \frac{2}{e\sigma}B_{\sigma}\int_{\Omega}\left\vert \nabla u\right\vert^{q(.)+\sigma}\,\mathrm{d}x\notag
	\\
	& \leq -2\int_{\Omega_2}\left\vert \nabla u\right\vert^{p_1}\,\mathrm{d}x + \frac{2}{e\sigma}B_{\sigma}\left(\int_{\Omega_1}\left\vert \nabla u\right\vert^{q_1+\sigma}\,\mathrm{d}x + \int_{\Omega_2}\left\vert \nabla u\right\vert^{q_2+\sigma}\,\mathrm{d}x\right) \notag
	\\
	& \leq -2\int_{\Omega_2}\left\vert \nabla u\right\vert^{p_1}\,\mathrm{d}x + \frac{4}{e\sigma}B_{\sigma}\left(\int_{\Omega_2}\left\vert \nabla u\right\vert^{q_1+\sigma}\,\mathrm{d}x + \int_{\Omega_2}\left\vert \nabla u\right\vert^{q_2+\sigma},\mathrm{d}x\right) \notag
\end{align}
By H\"{o}lder's inequality, we have
\begin{equation*}
	\left\vert \Omega \right\vert^{\frac{2-p_1}{2}}\left(\int_{\Omega_2}\left\vert \nabla u\right\vert^2\,\mathrm{d}x\right)^{\frac{p_1}{2}}\leq \int_{\Omega_2}\left\vert \nabla u\right\vert^{p_1}\,\mathrm{d}x,
\end{equation*}
while the Poincar\'{e}'s inequality gives
\begin{equation*}
	\int_{\Omega}\left\vert \nabla u\right\vert^2\,\mathrm{d}x\geq \lambda_1\int_{\Omega}\left\vert u\right\vert^2\,\mathrm{d}x,
\end{equation*}
where $\lambda_1$ is the first eigenvalue of the Dirichlet Laplacian on $\Omega$. Thus \eqref{eq:theta} satisfies the following inequality
\begin{align*}
	\theta(t) & \leq \left(\frac{1}{\lambda_1}+1\right) \int_{\Omega}\left\vert \nabla u\right\vert^2\,\mathrm{d}x
	\\
	&= \left(\frac{1}{\lambda_1}+1\right) \int_{\Omega_1}\left\vert \nabla u\right\vert^2\,\mathrm{d}x+\left(\frac{1}{\lambda_1}+1\right) \int_{\Omega_2}\left\vert \nabla u\right\vert^2\,\mathrm{d}x \leq 2\frac{\lambda_1+1}{\lambda_1}\int_{\Omega_2}\left\vert \nabla u\right\vert^2\,\mathrm{d}x.\notag
\end{align*}
Moreover, from the Sobolev embedding we have
\begin{align*}
	\int_{\Omega_2}\left\vert \nabla u\right\vert^{q_1+\sigma}\,\mathrm{d}x + \int_{\Omega_2}\left\vert \nabla u\right\vert^{q_2+\sigma}\,\mathrm{d}x & \leq \int_{\Omega}\left\vert \nabla u\right\vert^{q_1+\sigma}\,\mathrm{d}x + \int_{\Omega}\left\vert \nabla u\right\vert^{q_2+\sigma}\,\mathrm{d}x
	\\
	& \leq C_1^{q_1+\sigma}\theta(t)^{\frac{q_1+\sigma}{2}} + C_2^{q_2+\sigma}\theta(t)^{\frac{q_2+\sigma}{2}}.
\end{align*}
Thus, \eqref{eq:thetaDerEst} gives
\begin{align}\label{eq:thetaDerBound}
	\theta^{\prime}(t) & \leq -2\left\vert \Omega\right\vert^{\frac{2-p_1}{2}}\left(\frac{\lambda_1}{2(\lambda_1+1)}\right)^{\frac{p_1}{2}}\theta(t)^{\frac{p_1}{2}} + \frac{4}{e\sigma}B_{\sigma}\left(C_1^{q_1+\sigma}\theta(t)^{\frac{q_1+\sigma}{2}}+C_2^{q_2+\sigma}\theta(t)^{\frac{q_2+\sigma}{2}}\right) \notag
	\\
	& =-\theta(t)^{\frac{p_1}{2}}\left(\mathrm{A}-\mathrm{B}\theta(t)^{\frac{q_1-p_1+\sigma}{2}}\right) -\theta(t)^{\frac{p_1}{2}}\left(\mathrm{A}-\mathrm{C}\theta(t)^{\frac{q_2-p_1+\sigma}{2}}\right).
\end{align}
We now claim that
\begin{equation}\label{eq:thetaBound}
	\theta(t) <\min \left(\left(\frac{\mathrm{A}}{\mathrm{2B}}\right)^{\frac{2}{q_1-p_1+\sigma}},\left(\frac{\mathrm{A}}{\mathrm{2C}}\right)^{\frac{2}{q_2-p_1+\sigma}}\right), \quad\mbox{ for all }t\geq 0.
\end{equation}
Indeed, if we suppose \eqref{eq:thetaBound} is not satisfied, by the continuity of $\theta(t)$ and \eqref{eq:u0cond}, there exists $t_0>0$ such that
\begin{equation}\label{eq:theta1}
	\theta(t_0) =\min \left(\left(\frac{\mathrm{A}}{\mathrm{2B}}\right)^{\frac{2}{q_1-p_1+\sigma}},\left(\frac{\mathrm{A}}{\mathrm{2C}}\right)^{\frac{2}{q_2-p_1+\sigma}}\right),
\end{equation}
and
\begin{equation}\label{eq:theta2}
	\theta(t) <\min \left(\left(\frac{\mathrm{A}}{\mathrm{2B}}\right)^{\frac{2}{q_1-p_1+\sigma}},\left(\frac{\mathrm{A}}{\mathrm{2C}}\right)^{\frac{2}{q_2-p_1+\sigma}}\right)\quad \text{ for }0\leq t<t_0.
\end{equation}
Since we are assuming in \eqref{eq:u0cond} that
\begin{align*}
	\theta(0)<\min \left(\left(\frac{\mathrm{A}}{\mathrm{2B}}\right)^{\frac{2}{q_1-p_1+\sigma}},\left(\frac{\mathrm{A}}{\mathrm{2C}}\right)^{\frac{2}{q_2-p_1+\sigma}}\right),
\end{align*}
we get from \eqref{eq:theta1} that $\theta'(t)>0$ for $0<t<t_0$. On the other hand, \eqref{eq:thetaDerBound} and \eqref{eq:theta2} implies that $\theta^{\prime}(t)<0$ for $0<t<t_0$. We then found a contradiction, meaning that \eqref{eq:thetaBound} holds. In particular, we have that $\theta^{\prime}(t)<0$ for all $t\geq 0$. Hence, the $H_0^1(\Omega)$-norm of the solution to \eqref{eq:main_eq} decays in time and cannot blow-up.

To obtain the decay estimates \eqref{eq:decay1} and \eqref{eq:decay2}, without losing generality let us assume that $\theta(t)\neq 0$ for all $t\geq 0$. We are allowed to do that because, if there exists $t_0\geq 0$ such that $\theta(t_0)=0$ then, since $\theta(t)$ is non-negative and decreasing for all $t\geq 0$, we would also have
\begin{align*}
	\theta(t)=0 \quad \mbox{ for all }t\geq t_0
\end{align*}
and \eqref{eq:decay1} and \eqref{eq:decay2} would be trivially satisfied.

If, on the other hand, $\theta(t)\neq 0$ for all $t\geq 0$, then from the differential inequality \eqref{eq:thetaDerBound} and \eqref{eq:thetaBound} we can write
\begin{align}\label{eq:thetaFrac}
	1\leq -\frac{\theta^{\prime}(t)}{\theta(t)^{\frac{p_1}{2}}\left(\mathrm{A}-\mathrm{B}\theta(t)^{\frac{q_1-p_1+\sigma}{2}}\right)}.
\end{align}	
We shall now distinguish two cases:

\paragraph{\textbf{Case 1: $p_1>2$.}} Integrating \eqref{eq:thetaFrac} over the interval $(0,t)$, and taking into account that $\theta(t)\leq\theta(0)$ for all $t\geq 0$, we obtain
\begin{align}\label{eq:thetaInt1}
	t\leq -\int_{\theta(0)}^{\theta(t)}\frac{d\gamma}{\gamma^{\frac{p_1}{2}}\left(\mathrm{A}-\mathrm{B}\gamma^{\frac{q_1-p_1+\sigma}{2}}\right)} = \int_{\theta(t)}^{\theta(0)}\frac{d\gamma}{\gamma^{\frac{p_1}{2}}\left(\mathrm{A}-\mathrm{B}\gamma^{\frac{q_1-p_1+\sigma}{2}}\right)},
\end{align}
which is equivalent to
\begin{align}\label{eq:thetaInt2}
	 t\leq\frac{1}{\mathrm{A}}\int_{\theta(t)}^{\theta(0)}\left(\frac{1}{\gamma^{\frac{p_1}{2}}}+\frac{\mathrm{B}\gamma^{\frac{q_1-p_1+\sigma}{2}}}{\gamma^{\frac{p_1}{2}}\left(\mathrm{A}-\mathrm{B}\gamma^{\frac{q_1-p_1+\sigma}{2}}\right)}\right)d\gamma.
\end{align}
Moreover, since by \eqref{eq:thetaBound} we have
\begin{align*}
	\frac{\mathrm{B}\gamma^{\frac{q_1-p_1+\sigma}{2}}}{\mathrm{A}-\mathrm{B}\gamma^{\frac{q_1-p_1+\sigma}{2}}}\leq 1\quad\mbox{ for all }\gamma\in[\theta(t),\theta(0)],
\end{align*}
we obtain from \eqref{eq:thetaInt2} that
\begin{align*}
	t & \leq\frac{2}{\mathrm{A}}\int_{\theta(t)}^{\theta(0)}\frac{d\gamma}{\gamma^{\frac{p_1}{2}}} = \frac{4}{A(2-p_1)}\left(\theta(0)^{\frac{2-p_1}{2}}-\theta(t)^{\frac{2-p_1}{2}}\right) 
	\\
	&= \frac{4}{A(2-p_1)}\Big(\|u_0\|^{2-p_1}-\|u(.,t)\|^{2-p_1}\Big).
\end{align*}
From the above inequality, \eqref{eq:decay1} easily follows by means of simple algebraic manipulations.

\paragraph{\textbf{Case 2: $p_1=2$.}} When $p_1=2$, we obtain from \eqref{eq:thetaInt1} that
\begin{align}\label{eq:thetaInt3}
	t\leq \int_{\theta(t)}^{\theta(0)}\frac{d\gamma}{\gamma\left(\mathrm{A}-\mathrm{B}\gamma^{\frac{q_1-2+\sigma}{2}}\right)}.
\end{align}
For simplicity of notation, we shall denote
\begin{align}\label{eq:varpi}
	\varpi:= \frac{q_1-2+\sigma}{2}.
\end{align}
Then if we introduce the change of variables $\xi=\gamma^\varpi$, we can compute
\begin{align}\label{eq:thetaInt4}
	\int_{\theta(t)}^{\theta(0)}\frac{d\gamma}{\gamma\left(\mathrm{A}-\mathrm{B}\gamma^\varpi\right)} &= \frac 1\varpi\int_{\theta(t)^\varpi}^{\theta(0)^\varpi}\frac{d\xi}{\xi\left(\mathrm{A}-\mathrm{B}\xi\right)} 
	\\
	&= \frac{1}{\mathrm{A}\varpi}\ln\left(\frac{\theta(0)^\varpi}{\mathrm{A}-\mathrm{B}\theta(0)^\varpi} \frac{\mathrm{A}-\mathrm{B}\theta(t)^\varpi}{\theta(t)^\varpi}\right).\notag 
\end{align}
We then get from \eqref{eq:thetaInt3} and \eqref{eq:thetaInt4} that
\begin{align*}
	\frac{\theta(t)^\varpi}{\mathrm{A}-\mathrm{B}\theta(t)^\varpi}\leq \frac{\theta(0)^\varpi}{\mathrm{A}-\mathrm{B}\theta(0)^\varpi} e^{-\mathrm{A}\varpi t}.
\end{align*}
Since by \eqref{eq:u0cond} we have
\begin{align*}
	\mathrm{A}-\mathrm{B}\theta(t)^\varpi >0,
\end{align*}
we then obtain that
\begin{align*}
	\theta(t)^\varpi &\leq \left(\frac{\mathrm{A}\theta(0)^\varpi}{\mathrm{A}-\mathrm{B}\theta(0)^\varpi} - \frac{\mathrm{B}\theta(0)^\varpi}{\mathrm{A}-\mathrm{B}\theta(0)^\varpi}\theta(t)^\varpi\right) e^{-\mathrm{A}\varpi t} \leq \frac{\mathrm{A}\theta(0)^\varpi}{\mathrm{A}-\mathrm{B}\theta(0)^\varpi}  e^{-\mathrm{A}\varpi t}.
\end{align*}
Hence
\begin{align*}
	\theta(t) \leq \left(\frac{\mathrm{A}\theta(0)^\varpi}{\mathrm{A}-\mathrm{B}\theta(0)^\varpi}\right)^{\frac 1\varpi} e^{-\mathrm{A}t}.
\end{align*}

Since, by definition, $\theta(t)=\|u\|_{H_0^1(\Omega)}^2$, \eqref{eq:decay2} follows immediately from the above estimate and \eqref{eq:varpi}. Our proof is then concluded.
\end{proof}

\section{Numerical study}\label{sec:5}
In this section, we address a 2D numerical illustration of the problem 
\begin{align}\label{eq:mainnum_eq}
	\begin{cases}
		u_t-\Delta u_{t}-\func{div}\left(\left\vert \nabla u\right\vert^{p(x,y)-2}\nabla u\right) =|u|^{q(x,y)-2}u\ln (|u|), & (x,y,t)\in \Omega\times (0,T)
		\\
		u(x,y,t)=0, & (x,y,t)\in \partial \Omega\times (0,T)
		\\
		u(x,y,0)=u_0(x,y)\neq 0, & (x,y)\in \Omega
	\end{cases}
\end{align}
where $T>0$ and $u$ represents the velocity of a fluid in the domain $\Omega= [-1,1]^2$. 

Our objective is to represent numerically the decay rate obtained in Theorem \ref{thm:Theo3} in the case $p>2$. To this end, let us start by introducing the numerical scheme that we will employ for our experiments. First of all, for $n\in\mathbb{N}$, let us consider the uniform time-mesh
\begin{align*}
	t_n = n\delta t, \quad \delta t = \frac{T}{n}, 
\end{align*}
and define
\begin{align*}
	&u^n=u^n(x,y):=u(x,y,t_n)
	\\[5pt]
	&u^{n+1}_t:=\frac{u^{n+1}-u^{n}}{\Delta t}.
\end{align*}
We then approximate \eqref{eq:mainnum_eq} through
\begin{align}\label{eq:mainnum_eq_approx}
	\begin{cases}
		u^{n+1}_t - \Delta u_t^{n+1}-\func{div}\left(\left\vert \nabla u^n\right\vert^{p-2}\nabla u^n\right) = \left\vert u^n\right\vert^{q-2}u^n\ln (|u^n|) & \text{ in }\Omega 
		\\ 
		u^n=0 & \text{ on }\partial \Omega
		\\
		u^0=u_0, & \text{ in } \Omega
	\end{cases}
\end{align}

For the space-discretization, we use $\mathcal P_1$ Lagrangian Finite Elements implemented through the software \texttt{FreeFem++} \cite{Freefem}. To this end, we define a triangular mesh $\mathcal M_h$ of mesh-size $h$ over the domain $\Omega$, composed by 5000 triangles and 2601 vertices. Let $\mathcal V_h$ be the corresponding space of piece-wise linear continuous functions. Then, approximating \eqref{eq:mainnum_eq_approx} amounts at solving the following fully-discrete problem: find $u_h^n\in \mathcal V_h$ such that 
\begin{align}\label{eq:discrete}
	\int_\Omega\Big[u_{h,t}^{n+1}\omega +\nabla u_{h,t}^{n+1}\cdot\nabla \omega + \vert \nabla u_h^n & \vert^{p(x,y)-2} \nabla u_h^n\cdot\nabla \omega \Big]\, \mathrm{d}x\mathrm{d}y = \int_\Omega \left\vert u_h^n\right\vert^{q(x,y)-2}u_h^n\ln(|u_h^n|)\omega\, \mathrm{d}x\mathrm{d}y. 
\end{align}
In what follows, we consider the variable exponents 
\begin{align*}
	p(x,y) = 0.2\left\vert \lfloor x\rfloor \right\vert + 2.5 \quad\text{ and }\quad q(x,y) = 0.1\left\vert \lfloor x\rfloor \right\vert + 6,
\end{align*}%
where $\lfloor .\rfloor$ denotes the greatest integer function, so that
\begin{itemize}
	\item[1.] $2< 2.5 = p_1 < p_2 = 2.7 < 6 = q_1 < q_2 = 6.1 < p^{\ast }(x,y)=+\infty$.
	\item[2.] Condition \eqref{eq:sigma} is verified for all $\sigma>0$. In what follows, we will take $\sigma = 0.1$.
\end{itemize}
Moreover, we select the initial datum 
\begin{align*}
	u_0(x,y)=0.25e^{-x^2-y^2}.
\end{align*}
whose $H_0^1(\Omega)$-norm is 
\begin{align*}
	\|u_0\|_{H_0^1(\Omega)} = \left(\|u_0\|_{L^2(\Omega)}^2 + \|\nabla u_0\|_{L^2(\Omega)}^2\right)^{\frac 12} \sim 0.477344.
\end{align*}
Finally, the time horizon for our simulations will be $T=1$. 

Let us now present the results of our numerical experiments. First of all, we can see in Table \ref{tab:1} and Figure \ref{fig:initNorm} that, with this particular choice of the initial datum and variable exponents, the condition \eqref{eq:u0cond}, i.e.
\begin{align*}
	\|u_0\|_{H_0^1(\Omega)} < \min \left(\left(\frac{\mathrm{A}}{\mathrm{2B}}\right)^{\frac{1}{q_1-p_1+\sigma}},\left(\frac{\mathrm{A}}{\mathrm{2C}}\right)^{\frac{1}{q_2-p_1+\sigma}}\right)=: \mathrm{M}
\end{align*}
is satisfied. 
\begin{table}[h!]
	\centering 
	\begin{tabular}{|c|c|c|c|c|c|c|}
		\hline $t_n$ & 0 & 0.1 & 0.2 & 0.3 & 0.4 & 0.5
		\\
		\hline $\mathrm{M}$ & 0.671566 & 0.671906 & 0.672586 & 0.673228 & 0.673835 & 0.674407
		\\
		\hline $t_n$ & 0.6 & 0.7 & 0.8 & 0.9 & 1 &
		\\
		\hline $\mathrm{M}$  & 0.674946 & 0.675454 & 0.675932 & 0.676382 & 0.676804 &	
		\\
		\hline 
	\end{tabular}\caption{Values of $\mathrm{M}$ for $t\in [0,1]$.}\label{tab:1}
\end{table}

\begin{SCfigure}[0.5][h!]
	\centering
	\includegraphics[scale=0.5]{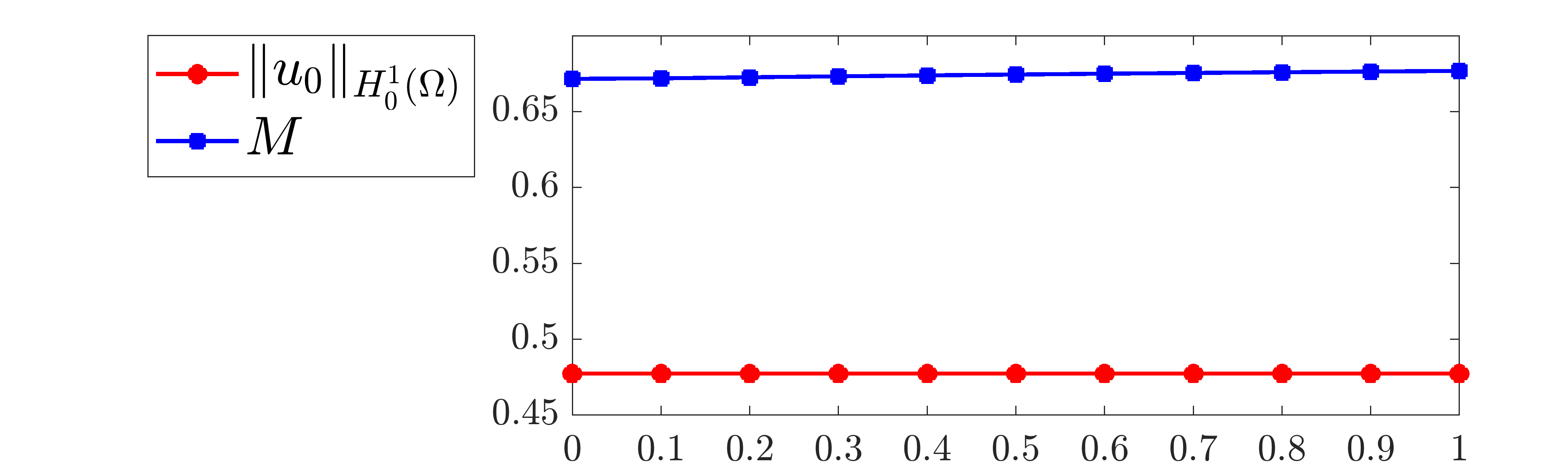}
	\caption{Fulfillment of the decay condition \eqref{eq:u0cond}.}\label{fig:initNorm}
\end{SCfigure}

Hence, according to Theorem \ref{thm:Theo3}, the solution to \eqref{eq:mainnum_eq} is expected to decay in $H_0^1(\Omega)$-norm at the rate \eqref{eq:decay1}, i.e.
\begin{align*}
	\|u_h^n\|_{H_0^1(\Omega)}\leq \left(\frac{4}{\mathrm{A}(p_1-2)t_n+4\left\Vert u_0\right\Vert_{H_0^1(\Omega)}^{2-p_1}}\right)^{\frac{1}{p_1-2}} =: \delta.
\end{align*} 
This is precisely what we can observe in Table \ref{tab:2} and Figures \ref{fig:1}, \ref{fig:2}.

$\newline$
\begin{table}[h!]
	\centering 
	\begin{tabular}{|c|c|c|c|c|c|c|}
		\hline $t_n$ & 0 & 0.1 & 0.2 & 0.3 & 0.4 & 0.5
		\\
		\hline $\|u_h^n\|_{H_0^1(\Omega)}$ & 0.477344 & 0.43491 & 0.420828 & 0.407467 & 0.394777 & 0.382713
		\\
		\hline $\delta$ & 0.477344 & 0.475403 & 0.473474 & 0.471557 & 0.469651 & 0.467757
		\\
		\hline $t_n$ & 0.6 & 0.7 & 0.8 & 0.9 & 1 &
		\\
		\hline $\|u_h^n\|_{H_0^1(\Omega)}$ & 0.371233 & 0.360298 & 0.349873 & 0.339927 & 0.330429 &	
		\\
		\hline $\delta$ & 0.465874 & 0.464003 & 0.462143 & 0.460294 & 0.458456 &
		\\
		\hline 
	\end{tabular}\caption{$H_0^1(\Omega)$-norm decay of the solution for $t\in [0,1]$.}\label{tab:2}
\end{table}

\begin{SCfigure}[0.5][h!]
\centering
	\includegraphics[scale=0.5]{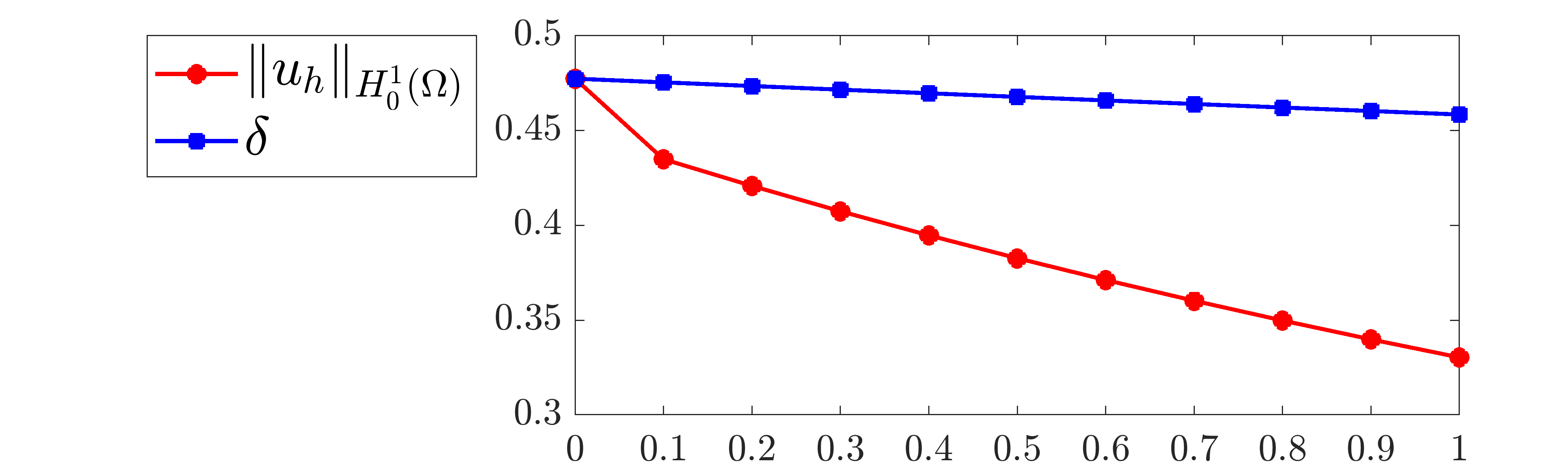}
	\caption{$H_0^1(\Omega)$-norm decay of the solution for $t\in [0,1]$.}\label{fig:1}
\end{SCfigure}

\begin{figure}[h]
	\centering
	\includegraphics[scale=0.3]{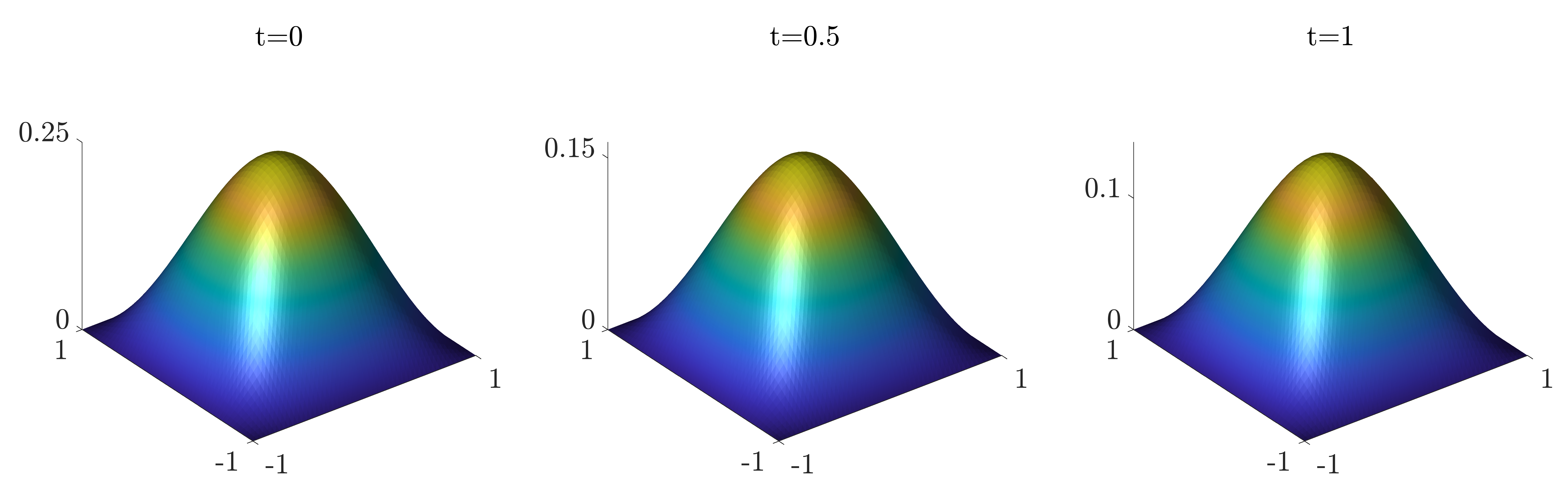}
	\caption{Time evolution of the solution of \eqref{eq:mainnum_eq} for $t\in [0,1]$.}\label{fig:2}
\end{figure}

All this confirms numerically the validity of our theoretical results presented in Section \ref{sec:4}.

\end{document}